\documentclass[12pt]{article}
\usepackage{amsmath,amssymb,amsthm, comment}

\newtheorem{theorem}{Theorem}[section]
\newtheorem{lemma}[theorem]{Lemma}
\newtheorem{corollary}[theorem]{Corollary}

\newtheoremstyle{named}{}{}{\itshape}{}{\bfseries}{.}{.5em}{\thmnote{#3's }#1}
\theoremstyle{named}

\newtheoremstyle{nnamed}{}{}{\itshape}{}{\bfseries}{.}{.5em}{\thmnote{#3' }#1}
\theoremstyle{nnamed}

\def\barr{\begin{array}}
\def\earr{\end{array}}

\title{On a conjecture by Haipeng Qu}
\author{Marius T\u arn\u auceanu}
\date{November 19, 2018}

\begin{document}

\maketitle

\begin{abstract}
In this note, we prove that $D_8\times C_2^{n-3}$ is the non-elementary abelian $2$-group of order $2^n$, $n\geq 3$,
whose number of subgroups of possible orders is maximal. This solves a conjecture by Haipeng Qu \cite{7}. A formula
for counting the subgroups of an (almost) extraspecial $2$-group is also presented.
\end{abstract}
\medskip

{\small
\noindent
{\bf MSC2000\,:} Primary 20D30; Secondary 20D60, 20D99.

\noindent
{\bf Key words\,:} subgroup lattice, (elementary) abelian $2$-group, (almost/generalized) extraspecial $2$-group, Goursat's lemma, central product.}

\section{Introduction}

Let $G$ be a finite $p$-group of order $p^n$. For $k=0,1,...,n$, we denote by $s_k(G)$ the number of subgroups of order $p^k$ of $G$. The starting point for our discussion is given by Theorem 1.4 of \cite{7} which proves that if $p$ is odd and $G$ is non-elementary abelian then
\begin{equation}
s_k(G)\leq s_k(M_p\times C_p^{n-3}), \forall\, 0\leq k\leq n,
\end{equation}where
\begin{equation}
M_p=\langle a,b \mid a^p=b^p=c^p=1, [a,b]=c, [c,a]=[c,b]=1\rangle\nonumber
\end{equation}is the non-abelian group of order $p^3$ and exponent $p$. Moreover, in \cite{7} it is conjectured that the inequalities (1) also hold for $p=2$. Note that in this case there is no direct analogue of $M_p$, as a group of exponent $2$ is abelian. But the dihedral group of order $8$ is close, as it is at least generated by elements of order $2$.

In the current note we prove the inequalities (1) for $p=2$ by replacing $M_p$ with $D_8$. This completes the work of Haipeng Qu. We also give a formula for the number of subgroups of an (almost) extraspecial $2$-group depending on the number of elementary abelian subgroups of possible orders. Our main result is the following.

\begin{theorem}
Let $G$ be a finite non-elementary abelian $2$-group of order $2^n$, $n\geq 3$. Then
\begin{equation}
s_k(G)\leq s_k(D_8\times C_2^{n-3}), \forall\, 0\leq k\leq n.
\end{equation}
\end{theorem}

Throughout this paper, we will use $A*B$ to denote the (amalgamated) central product of the groups $A$ and $B$ having isomorphic centres, and $X^{*r}$ to denote the central product of $r$ copies of the group $X$. Also, we will denote by $\binom{n}{k}_p$ the number of subgroups of order $p^k$ in an elementary abelian $p$-group of order $p^n$.

We recall several basic definitions and results on $2$-groups that will be useful to us. A finite $2$-group is called:
\begin{itemize}
\item[-] \textit{extraspecial} if $Z(G)=G'=\Phi(G)$ has order $2$;
\item[-] \textit{almost extraspecial} if $G'=\Phi(G)$ has order $2$ and $Z(G)\cong C_4$;
\item[-] \textit{generalized extraspecial} if $G'=\Phi(G)$ has order $2$ and $G'\subseteq Z(G)$.
\end{itemize}The structure of these groups is well-known (see e.g. Theorem 2.3 of \cite{2} and Lemma 3.2 of \cite{8}):

\begin{lemma}
Let $G$ be a finite $2$-group.
\begin{itemize}
\item[{\rm a)}] If $G$ is extraspecial, then $|G|=2^{2r+1}$ for some positive integer $r$ and either $G\cong D_8^{*r}$ or $G\cong Q_8*D_8^{*(r-1)}$.
\item[{\rm b)}] If $G$ is almost extraspecial, then $|G|=2^{2r+2}$ for some positive integer $r$ and $G\cong D_8^{*r}*C_4$.
\item[{\rm c)}] If $G$ is generalized extraspecial, then either $G\cong E\times A$ or $G\cong (E*C_4)\times A$, where $E$ is an extraspecial $2$-group and $A$ is an elementary abelian $2$-group.
\end{itemize}
\end{lemma}The key result of \cite{7} is Theorem 1.3. For $p=2$, it states the following.

\begin{lemma}
Let $G$ be a finite $2$-group of order $2^n$ and $M$ be a normal subgroup of order $2$ of $G$. Then
\begin{equation}
s_k(G)\leq s_k(G/M\times C_2), \forall\, 0\leq k\leq n.\nonumber
\end{equation}
\end{lemma}We also present a result that follows from Corollary 2 of \cite{3}.

\begin{lemma}
Let $G$ be a finite non-elementary abelian $2$-group of order $2^n$, $n\geq 3$, and $L_1(G)$ be the set of cyclic subgroups of $G$. Then
\begin{equation}
|L_1(G)|\leq 7\cdot 2^{n-3}=|L_1(D_8\times C_2^{n-3})|.\nonumber
\end{equation}
\end{lemma}

\section{Proofs of the main results}

First of all, we recall the well-known Goursat’s lemma (see e.g. (4.19) of \cite{9}, I) that will be intensively used in our proofs.

\begin{theorem}
Let $A$ and $B$ be two finite groups. Then every subgroup $H$ of the direct product $A\times B$ is completely determined by a quintuple $(A_1,A_2,B_1,B_2,\newline\varphi)$, where $A_1\unlhd A_2\leq A$, $B_1\unlhd B_2\leq B$ and $\varphi:A_2/A_1\longrightarrow B_2/B_1$ is an isomorphism, more exactly $H=\{(a,b)\in A_2\times B_2\mid \varphi(aA_1)=bB_1\}$. Moreover, we have $|H|=|A_1||B_2|=|A_2||B_1|$.
\end{theorem}

Next we remark that Lemma 1.3 can be easily generalized in the following way.

\begin{lemma}
Let $G$ be a finite $2$-group of order $2^n$ and $M$ be a normal subgroup of order $2^r$ of $G$. Then
\begin{equation}
s_k(G)\leq s_k(G/M\times C_2^r), \forall\, 0\leq k\leq n.\nonumber
\end{equation}
\end{lemma}

\begin{proof}
Take a chief series of $G$ containing $M$ and use induction on $r$.
\end{proof}

The following lemma shows that the inequalities (2) hold for finite abelian $2$-groups.

\begin{lemma}
Let $G$ be a finite abelian $2$-group of order $2^n$, $n\geq 3$. If $G$ is non-elementary abelian, then
\begin{equation}
s_k(G)\leq s_k(D_8\times C_2^{n-3}), \forall\, 0\leq k\leq n.\nonumber
\end{equation}
\end{lemma}

\begin{proof}
Since $G$ is not elementary abelian, we infer that it has a subgroup $M$ of order $2^{n-2}$ such that $G/M\cong C_4$. Then Lemma 2.2 leads to
\begin{equation}
s_k(G)\leq s_k(C_4\times C_2^{n-2}), \forall\, 0\leq k\leq n,\nonumber
\end{equation}and so it suffices to prove that
\begin{equation}
s_k(G_1\times C_2^{n-3})\leq s_k(D_8\times C_2^{n-3}), \forall\, 0\leq k\leq n,
\end{equation}where $G_1=C_4\times C_2$. By Theorem 2.1, we know that a subgroup of order $2^k$ of $G_1\times C_2^{n-3}$ is completely determined by a quintuple $(A_1,A_2,B_1,B_2,\varphi)$, where $A_1\unlhd A_2\leq G_1$, $B_1\unlhd B_2\leq C_2^{n-3}$, $\varphi:A_2/A_1\longrightarrow B_2/B_1$ is an isomorphism, and $|A_2||B_1|=2^k$. Note that $\varphi$ can be chosen in a unique way for the elementary abelian sections of orders $1$ and $2$ of $G_1$, and in $6=|{\rm Aut}(C_2^2)|$ ways for the elementary abelian sections of order $4$ of $G_1$. We distinguish the following cases:
\begin{itemize}
\item[{\rm a)}] $A_2=1$.\\
Then $A_1=1$ and $B_1=B_2$ is one of the $\binom{n-3}{k}_2$ subgroups of order $2^k$ of $C_2^{n-3}$. Clearly, these determine $\binom{n-3}{k}_2$ distinct subgroups of $G_1\times C_2^{n-3}$.
\item[{\rm b)}] $A_2$ is one of the three subgroups of order $2$ of $G_1$.\\
Then $B_1$ is of order $2^{k-1}$ and can be chosen in $\binom{n-3}{k-1}_2$ ways. If $A_1=A_2$ then $B_2=B_1$, while if $A_1=1$ then $B_2$ is one of the $2^{n-k-2}-1$ subgroups of order $2^k$ of $C_2^{n-3}$ containing $B_1$. So, in this case we have
\begin{equation}
3\cdot\binom{n-3}{k-1}_2+3\cdot(2^{n-k-2}-1)\binom{n-3}{k-1}_2=3\cdot 2^{n-k-2}\binom{n-3}{k-1}_2\nonumber
\end{equation}distinct subgroups of $G_1\times C_2^{n-3}$.
\item[{\rm c)}] $A_2$ is one of the two cyclic subgroups of order $4$ of $G_1$.\\
Then $B_1$ is of order $2^{k-2}$ and can be chosen in $\binom{n-3}{k-2}_2$ ways. If $A_1=A_2$ then $B_2=B_1$, while if $|A_1|=2$ then $B_2$ is one of the $2^{n-k-1}-1$ subgroups of order $2^{k-1}$ of $C_2^{n-3}$ containing $B_1$. So, in this case we have
\begin{equation}
2\cdot\binom{n-3}{k-2}_2+2\cdot(2^{n-k-1}-1)\binom{n-3}{k-1}_2=2^{n-k}\binom{n-3}{k-2}_2\nonumber
\end{equation}distinct subgroups of $G_1\times C_2^{n-3}$.
\item[{\rm d)}] $A_2$ is the unique subgroup isomorphic to $C_2^2$ of $G_1$.\\
Again, $B_1$ is of order $2^{k-2}$ and can be chosen in $\binom{n-3}{k-2}_2$ ways. If $A_1=A_2$ then $B_2=B_1$; if $|A_1|=2$ then $B_2$ is one of the $2^{n-k-1}-1$ subgroups of order $2^{k-1}$ of $C_2^{n-3}$ containing $B_1$; if $A_1=1$ then $B_2$ is one of the $\binom{n-k-1}{2}_2$ subgroups of order $2^k$ of $C_2^{n-3}$ containing $B_1$. One obtains
$$\binom{n-3}{k-2}_2\hspace{-2mm}+3(2^{n{-}k{-}1}{-}1)\binom{n-3}{k-2}_2\hspace{-2mm}+6\binom{n{-}k{-}1}{2}_2\hspace{-1mm}\binom{n-3}{k-2}_2\hspace{-3mm}=\hspace{-1mm}2^{2n-2k-2}\binom{n-3}{k-2}_2$$distinct subgroups of $G_1\times C_2^{n-3}$.
\item[{\rm e)}] $A_2=G_1$.\\
In this case $B_1$ is of order $2^{k-3}$ and can be chosen in $\binom{n-3}{k-3}_2$ ways. If $A_1=A_2$ then $B_2=B_1$; if $|A_1|=4$ then $B_2$ is one of the $2^{n-k}-1$ subgroups of order $2^{k-2}$ of $C_2^{n-3}$ containing $B_1$; if $A_1$ is the unique subgroup of order $2$ of $G_1$ such that $A_2/A_1\cong C_2^2$ then $B_2$ is one of the $\binom{n-k}{2}_2$ subgroups of order $2^{k-1}$ of $C_2^{n-3}$ containing $B_1$. One obtains
$$\binom{n-3}{k-3}_2\hspace{-2mm}+3(2^{n-k}-1)\binom{n-3}{k-3}_2\hspace{-2mm}+6\binom{n-k}{2}_2\hspace{-1mm}\binom{n-3}{k-3}_2\hspace{-3mm}=\hspace{-1mm}2^{2n-2k}\binom{n-3}{k-3}_2$$distinct subgroups of $G_1\times C_2^{n-3}$.
\end{itemize}

By summing all above quantities, we get
$$s_k(G_1\times C_2^{n-3})=\binom{n-3}{k}_2+3\cdot 2^{n-k-2}\binom{n-3}{k-1}_2+2^{n-k}(2^{n-k-2}+1)\binom{n-3}{k-2}_2+$$ $$\hspace{-45mm}2^{2n-2k}\binom{n-3}{k-3}_2.$$A similar computation leads to
$$s_k(D_8\times C_2^{n-3})=\binom{n-3}{k}_2+5\cdot 2^{n-k-2}\binom{n-3}{k-1}_2+2^{n-k-1}(2^{n-k}+1)\binom{n-3}{k-2}_2+$$ $$\hspace{-45mm}2^{2n-2k}\binom{n-3}{k-3}_2.$$It is now clear that the inequalities (3) are true, completing the proof.
\end{proof}

In what follows we will focus on describing the subgroup lattice of an (almost) extraspecial $2$-group $G$. This can be easily made by using Lemma 2.6 of \cite{2}.

\begin{lemma}
If $G$ is an (almost) extraspecial $2$-group, then $L(G)$ consists of:
\begin{itemize}
\item[{\rm a)}] the trivial subgroup;
\item[{\rm b)}] all subgroups containing $\Phi(G)$; moreover, these are the normal non-trivial subgroups of $G$;
\item[{\rm c)}] all complements of $\Phi(G)$ in the elementary abelian subgroups of order $\geq 4$ containing $\Phi(G)$; moreover,
\begin{itemize}
\item[-] $\Phi(G)$ has $2^i$ complements in an elementary abelian subgroup of order $2^{i+1}$ containing $\Phi(G)$;
\item[-] two non-normal subgroups $H$ and $K$ of $G$ are conjugate if and only if $H\Phi(G)=K\Phi(G)$;
\item[-] given two non-normal subgroups $H$ and $K$ of $G$, if $H^x\subseteq K$ for some $x\in [G/N_G(H)]$, then $x$ is the unique element of $[G/N_G(H)]$ with this property.
\end{itemize}
\end{itemize}
\end{lemma}

It is well-known that the order of maximal elementary abelian subgroups of an extraspecial $2$-group of order $2^{2r+1}$ or of an almost extraspecial $2$-group of order $2^{2r+2}$ is $2^{r+1}$. Let us denote by $e_i(G)$ the number of elementary abelian subgroups of order $2^i$ containing $\Phi(G)$, $i=2,3,...,r+1$. Under this notation, a formula for the number of subgroups of $G$ can be inferred from Lemma 2.4.

\begin{corollary}
If $G$ is an extraspecial $2$-group of order $2^{2r+1}$, then
\begin{equation}
|L(G)|=1+\sum_{i=0}^{2r}\binom{2r}{i}_2+\sum_{i=1}^{r}e_{i+1}(G)2^i,
\end{equation}while if $G$ is an almost extraspecial $2$-group of order $2^{2r+2}$, then
\begin{equation}
|L(G)|=1+\sum_{i=0}^{2r+1}\binom{2r+1}{i}_2+\sum_{i=1}^{r}e_{i+1}(G)2^i.
\end{equation}
\end{corollary}

In the particular cases $r=2$ and $r=1$ the equalities (4) and (5), respectively, lead to some known results (see e.g. Chapter 3.3 of \cite{5} and Example 4.5 of \cite{6}).

\bigskip\noindent{\bf Examples.}
\begin{itemize}
\item[{\rm a)}] For $G=D_8*D_8$ we obtain $e_2(G)=9$ and $e_3(G)=6$, implying that $$|L(G)|=1+\sum_{i=0}^{4}\binom{4}{i}_2+2\cdot 9+4\cdot 6=110.$$
\item[{\rm b)}] For $G=Q_8*D_8$ we obtain $e_2(G)=5$ and $e_3(G)=0$, implying that $$|L(G)|=1+\sum_{i=0}^{4}\binom{4}{i}_2+2\cdot 5=78.$$
\item[{\rm c)}] For $G=D_8*C_4$ we obtain $e_2(G)=3$, and so $$|L(G)|=1+\sum_{i=0}^{3}\binom{3}{i}_2+2\cdot 3=23.$$
\end{itemize}

\noindent{\bf Remark.} As we have seen above, the subgroup structure of an (almost) extraspecial $2$-group $G$ depends on its elementary abelian subgroups containing $\Phi(G)=\langle x\rangle$. We remark that these are the totally singular subspaces of $G/\Phi(G)$ with respect to the quadratic form $q:G/\Phi(G)\longrightarrow\mathbb{F}_2$, where $q(\bar{v})$ is the element $a\in\mathbb{F}_2$ such that $v^2=x^a$, $\forall\, \bar{v}\in G/\Phi(G)$.
\bigskip

Next we will compare the subgroup lattices of an (almost) extraspecial $2$-group $G$ of order $2^n$, $n\geq 3$, and of $D_8\times C_2^{n-3}$. Since $\Phi(D_8\times C_2^{n-3})$ is of order $2$, $D_8\times C_2^{n-3}/\Phi(D_8\times C_2^{n-3})$ and $G/\Phi(G)$ have the same dimension over $\mathbb{F}_2$, namely $n-1$. Also, we note that the order of maximal elementary abelian subgroups of $D_8\times C_2^{n-3}$ is $2^{n-1}$, which is greater or equal to the order of maximal elementary abelian subgroups of $G$. We infer that
\begin{equation}
|L(D_8\times C_2^{n-3})|=1+\sum_{i=0}^{n-1}\binom{n-1}{i}_2+\sum_{i=1}^{n-2}e_{i+1}(D_8\times C_2^{n-3})2^i.\nonumber
\end{equation}

The following lemma will be crucial in our proof.

\begin{lemma}
Under the above notation, we have $e_2(G)\leq e_2(D_8\times C_2^{n-3})$.
\end{lemma}

\begin{proof}
By Lemma 1.4, we know that
\begin{equation}
|L_1(G)|\leq |L_1(D_8\times C_2^{n-3})|.\nonumber
\end{equation}Let $c_i(G)$ be the number of cyclic subgroups of order $2^i$ of $G$. Since both $G$ and $D_8\times C_2^{n-3}$ are of exponent $4$, the above inequality can be rewritten as
\begin{equation}
1+c_2(G)+c_4(G)\leq 1+c_2(D_8\times C_2^{n-3})+c_4(D_8\times C_2^{n-3}).\nonumber
\end{equation}On the other hand, we have
\begin{equation}
2^n=1+c_2(G)+2c_4(G)=1+c_2(D_8\times C_2^{n-3})+2c_4(D_8\times C_2^{n-3})\nonumber
\end{equation}and consequently
\begin{equation}
2^n-c_4(G)\leq 2^n-c_4(D_8\times C_2^{n-3}),\nonumber
\end{equation}i.e.
\begin{equation}
c_4(D_8\times C_2^{n-3})\leq c_4(G).
\end{equation}It is clear that the cyclic subgroups of order $4$ of these groups contain the Frattini subgroup. Thus (6) leads to $$e_2(G)\leq e_2(D_8\times C_2^{n-3}),$$as desired.
\end{proof}

Since any elementary abelian subgroup of $G$ containing $\Phi(G)$ is a direct sum of elementary abelian subgroups of order $4$, from Lemma 2.6 we infer that
$$e_i(G)\leq e_i(D_8\times C_2^{n-3}), \mbox{ for all } i.$$An immediate consequence of this fact is that the inequalities (2) also hold for (almost) extraspecial $2$-groups.

\begin{corollary}
If $G$ is an (almost) extraspecial $2$-group of order $2^n$, $n\geq 3$, then
\begin{equation}
s_k(G)\leq s_k(D_8\times C_2^{n-3}), \forall\, 0\leq k\leq n.\nonumber
\end{equation}
\end{corollary}

A similar thing can be said about elementary abelian sections of $G$ and of $D_8\times C_2^{n-3}$.

\begin{corollary}
Given a $2$-group $G$, we denote by $\mathcal{S}_{(\alpha,\beta)}(G)$ the set of all elementary abelian sections $H_2/H_1\cong C_2^{\alpha}$ of $G$ with $|H_1|=2^{\beta}$. If $G$ is (almost) extraspecial of order $2^n$, then
\begin{equation}
|\mathcal{S}_{(\alpha,\beta)}(G)|\leq |\mathcal{S}_{(\alpha,\beta)}(D_8\times C_2^{n-3})|, \mbox{ for all } \alpha \mbox{ and } \beta.\nonumber
\end{equation}
\end{corollary}

\begin{proof}
Write $$\mathcal{S}_{(\alpha,\beta)}(G)=\mathcal{S}_{(\alpha,\beta)}^1(G)\cup\mathcal{S}_{(\alpha,\beta)}^2(G)\cup\mathcal{S}_{(\alpha,\beta)}^3(G)\cup\mathcal{S}_{(\alpha,\beta)}^4(G),$$where
\begin{itemize}
\item[-] $\mathcal{S}_{(\alpha,\beta)}^1(G)=\{H_2/H_1\in\mathcal{S}_{(\alpha,\beta)}(G)\mid \Phi(G)\subseteq H_1\}$,
\item[-] $\mathcal{S}_{(\alpha,\beta)}^2(G)=\{H_2/H_1\in\mathcal{S}_{(\alpha,\beta)}(G)\mid H_1=1\}$,
\item[-] $\mathcal{S}_{(\alpha,\beta)}^3(G)=\{H_2/H_1\in\mathcal{S}_{(\alpha,\beta)}(G)\mid \Phi(G)\nsubseteq H_2, H_1\neq 1\}$,
\item[-] $\mathcal{S}_{(\alpha,\beta)}^4(G)=\{H_2/H_1\in\mathcal{S}_{(\alpha,\beta)}(G)\mid \Phi(G)\subseteq H_2, \Phi(G)\nsubseteq H_1, H_1\neq 1\}$.
\end{itemize}

Since $G/\Phi(G)\cong D_8\times C_2^{n-3}/\Phi(D_8\times C_2^{n-3})$, we have
\begin{equation}
|\mathcal{S}_{(\alpha,\beta)}^1(G)|=|\mathcal{S}_{(\alpha,\beta)}^1(D_8\times C_2^{n-3})|.
\end{equation}

Clearly, $\mathcal{S}_{(\alpha,\beta)}^2(G)=\mathcal{S}_{(\alpha,0)}(G)$ is the number of elementary abelian subgroups of order $2^{\alpha}$ of $G$,
and so
\begin{equation}
|\mathcal{S}_{(\alpha,\beta)}^2(G)|=e_{\alpha}(G)+e_{\alpha+1}(G)2^{\alpha},\nonumber
\end{equation}implying that
\begin{equation}
|\mathcal{S}_{(\alpha,\beta)}^2(G)|\leq|\mathcal{S}_{(\alpha,\beta)}^2(D_8\times C_2^{n-3})|.
\end{equation}

We observe that every section $H_2/H_1\in\mathcal{S}_{(\alpha,\beta)}^3(G)$ determines a section $H_2\Phi(G)/H_1\Phi(G)\in\mathcal{S}_{(\alpha,\beta+1)}^1(G)$ with $H_2\Phi(G)\cong C_2^{\alpha+\beta+1}$ and $H_1\Phi(G)\cong C_2^{\beta+1}$. Conversely, every section $A/B\in\mathcal{S}_{(\alpha,\beta+1)}^1(G)$ with $A\cong C_2^{\alpha+\beta+1}$ and $B\cong C_2^{\beta+1}$ determines $2^{\beta}$ sections $H_2/H_1\in\mathcal{S}_{(\alpha,\beta)}^3(G)$. Since there are $e_{\alpha+\beta+1}(G)\binom{\alpha+\beta+1}{\beta+1}_2$ such sections $A/B\in\mathcal{S}_{(\alpha,\beta+1)}^1(G)$, we infer that
\begin{equation}
|\mathcal{S}_{(\alpha,\beta)}^3(G)|=e_{\alpha+\beta+1}(G)\binom{\alpha+\beta+1}{\beta+1}_22^{\beta}\leq|\mathcal{S}_{(\alpha,\beta)}^3(D_8\times C_2^{n-3})|.
\end{equation}

Let $H_2/H_1\in\mathcal{S}_{(\alpha,\beta)}^4(G)$. Then $\Phi(H_2)\subseteq H_1$. On the other hand, we have $\Phi(H_2)\subseteq \Phi(G)$ because $H_2$ is normal in $G$. Thus $\Phi(H_2)=1$, i.e. $H_2$ is elementary abelian, and it can be chosen in $e_{\alpha+\beta}(G)$ ways. Also, $H_1$ is a complement of $\Phi(G)$ in one of the $\binom{\alpha+\beta}{\beta+1}_2$ subgroups of order $2^{\beta+1}$ of $H_2$, and it can be chosen in $\binom{\alpha+\beta}{\beta+1}_22^{\beta}$ ways. Hence
\begin{equation}
|\mathcal{S}_{(\alpha,\beta)}^4(G)|=e_{\alpha+\beta}(G)\binom{\alpha+\beta}{\beta+1}_22^{\beta}\leq|\mathcal{S}_{(\alpha,\beta)}^4(D_8\times C_2^{n-3})|.
\end{equation}

Obviously, the relations (7), (8), (9) and (10) lead to $$|\mathcal{S}_{(\alpha,\beta)}(G)|\leq |\mathcal{S}_{(\alpha,\beta)}(D_8\times C_2^{n-3}),$$as desired.
\end{proof}

We are now able to prove our main result.

\bigskip\noindent{\bf Proof of Theorem 1.1.} Let $|G'|=2^m$. If $m=0$, then $G$ is abelian and the conclusion follows from Lemma 2.3. Assume that $m\geq 1$.

If $G/G'$ is not elementary abelian, then
\begin{equation}
s_k(G)\leq s_k(G/G'\times C_2^m)\leq s_k(D_8\times C_2^{n-3}), \forall\, 0\leq k\leq n,\nonumber
\end{equation}where the first inequality is obtained by Lemma 2.2, while the second one by Lemma 2.3.

If $G/G'$ is elementary abelian, then $G'=\Phi(G)$. Let $M$ be a normal subgroup of $G$ such that $M\subseteq G'$ and $[G':M]=2$. Then
\begin{equation}
s_k(G)\leq s_k(G_1\times C_2^{m-1}), \forall\, 0\leq k\leq n,\nonumber
\end{equation}where $G_1=G/M$ satisfies the conditions $$G_1'=\Phi(G_1), |G_1'|=2 \mbox{ and } G_1'\subseteq Z(G_1),$$i.e. it is a generalized extraspecial $2$-group. Then either $$G_1\cong E\times A \mbox{ or } G_1\cong (E*C_4)\times A,$$where $E$ is an extraspecial $2$-group and $A$ is an elementary abelian $2$-group, by Lemma 1.2. In other words, $G_1$ is a direct product of an (almost) extraspecial $2$-group and an elementary abelian $2$-group. So, it suffices to prove that if $G_2$ is an (almost) extraspecial $2$-group of order $2^q$, then
\begin{equation}
s_k(G_2\times C_2^{n-q})\leq s_k((D_8\times C_2^{q-3})\times C_2^{n-q}), \forall\, 0\leq k\leq n.
\end{equation}Let $A$ be one of the groups $G_2$ or $D_8\times C_2^{q-3}$. From Theorem 2.1 it follows that a subgroup $H\leq A\times C_2^{n-q}$ of order $2^k$ is completely determined by a quintuple $(A_1,A_2,B_1,B_2,\varphi)$, where $A_1\unlhd A_2\leq A$, $B_1\unlhd B_2\leq C_2^{n-q}$, $\varphi:A_2/A_1\longrightarrow B_2/B_1$ is an isomorphism, and $|A_2||B_1|=2^k$. By fixing the section $B_2/B_1$ of $C_2^{n-q}$ and $\varphi\in{\rm Aut}(B_2/B_1)$, we infer that $H$ depends only on the choice of the section $A_2/A_1\in\mathcal{S}_{(\alpha,\beta)}(A)$, where $(\alpha,\beta)\in\{(i,j)\mid i=0,1,...,k,\, j=0,1,...,k-i\}$. So, to prove the inequalities (11) it suffices to compare the numbers of elementary abelian sections $A_2/A_1$ of the two groups $G_2$ and $D_8\times C_2^{q-3}$, where $|A_1|$ and $|A_2|$ are arbitrary. It is now clear that the conclusion follows from Corollary 2.8, completing the proof.\qed
\bigskip

Finally, we mention that a result similar with Lemma 1.4 can be obtained from Theorem 1.1.

\begin{corollary}
Let $G$ be a finite non-elementary abelian $2$-group of order $2^n$, $n\geq 3$. Then
\begin{equation}
|L(G)|\leq |L(D_8\times C_2^{n-3})|.\nonumber
\end{equation}
\end{corollary}

\bigskip\noindent{\bf Acknowledgements.} The author is grateful to the reviewer for its remarks
which improve the previous version of the paper.

\vspace*{3ex}\small

\hfill
\begin{minipage}[t]{5cm}
Marius T\u arn\u auceanu \\
Faculty of  Mathematics \\
``Al.I. Cuza'' University \\
Ia\c si, Romania \\
e-mail: {\tt tarnauc@uaic.ro}
\end{minipage}

\end{document}